\documentclass{amsart}

\usepackage[T1]{fontenc}
\usepackage{lmodern}
\usepackage{geometry}
\usepackage{listings}
\usepackage{amsmath,amssymb,amsthm,amsfonts,mathtools}
\usepackage{hyperref}

\DeclareMathOperator{\dist}{dist}

\hypersetup{
  colorlinks=true,
  linkcolor=blue,
  citecolor=blue,
  urlcolor=blue,
  pdftitle={A Spectral Correction for Fourier-Weighted Estimates and a Sharp L2 Constant},
  pdfauthor={David Kalaj}
}

\newcommand{\R}{\mathbb{R}}
\newcommand{\C}{\mathbb{C}}
\newcommand{\dd}{\,\mathrm{d}}
\newcommand{\abs}[1]{\left|#1\right|}

\newcommand{\ip}[2]{\left\langle #1,#2\right\rangle}

\theoremstyle{plain}
\newtheorem{theorem}{Theorem}
\newtheorem{lemma}[theorem]{Lemma}

\theoremstyle{remark}
\newtheorem{remark}[theorem]{Remark}
\newtheorem{example}[theorem]{Example}

\title{The $L^2$ Norm of the Interior Cauchy Transform: Beyond the First Dirichlet Eigenvalue}
\author{David Kalaj}
\address{University of Montenegro, Faculty of Natural Sciences and
Mathematics, Cetinjski put b.b. 81000 Podgorica, Montenegro}
\email{davidk@ucg.ac.me}
\date{}

\begin{document}

\begin{abstract}
We study sharp \(L^2\) bounds for the interior Cauchy transform \(C_D\) on a bounded planar domain \(D\) and clarify its connection with the Dirichlet spectrum. We analyze an approach that replaces fractional Dirichlet powers on \(D\) by Euclidean Fourier multipliers after extension by zero, and show that this substitution can change the optimal constants. In particular, we construct an explicit endpoint counterexample on the unit disk to a Fourier-weighted inequality appearing in \cite{Dostanic1996}. This identifies a gap in the derivation of the conjectured identity \(\|C_D\|_{L^2\to L^2}=2/\sqrt{\lambda_1(D)}\).

We then identify the correct sharp constant for the endpoint Fourier weight $|\xi|^{-1}$ in terms of the top eigenvalue of a natural positive potential-type operator on $D$. Finally, we show that testing $C_D$ on the first Dirichlet eigenfunction already exceeds the spectral threshold, so $\|C_D\|_{L^2\to L^2}>2/\sqrt{\lambda_1(D)}$ for simply connected domains and also for annuli $A(r,R)$, and we prove a rigidity result: equality with the spectral value occurs if and only if $D$ is a disk.
\end{abstract}

\maketitle


\section{Introduction}\label{sec:introduction}

Let $D\subset\C$ be a bounded domain. The (interior) Cauchy transform
\begin{equation}\label{eq:intro-cauchy}
(C_D f)(z):=\frac{1}{\pi}\int_D \frac{f(w)}{z-w}\,dA(w),\qquad z\in D,
\end{equation}
is a basic singular integral operator in planar function theory and provides a
canonical $L^2$--solution operator to $\bar\partial u=f$ on $D$. It also appears
naturally alongside Bergman-type projections and in questions on spectral and
smoothing properties of integral operators (see, for instance,
\cite{AndersonKhavinsonLomonosov1992,Dostanic1996,Kalaj2012AIM,ZhuKalaj2020}).

A fundamental point is that the global Cauchy transform on $\C$ is not bounded on
$L^2(\C)$, yet its restriction to a bounded domain may admit sharp $L^2$ bounds.
In their seminal work \cite{AndersonHinkkanen1989}, Anderson and Hinkkanen
considered the unit disk $\mathbb{D}$ and proved that if $f\in L^2(\mathbb{D})$ is
extended by zero outside $\mathbb{D}$, then $C_\mathbb{D} f\in L^2(\mathbb{D})$ and
\begin{equation}\label{eq:AH}
\|C_\mathbb{D} f\|_{L^2(\mathbb{D})}\le \frac{2}{j_{0,1}}\|f\|_{L^2(\mathbb{D})},
\end{equation}
where $j_{0,1}$ is the first positive zero of the Bessel function $J_0$; moreover,
the constant $2/j_{0,1}$ is sharp. Since $\lambda_1(\mathbb{D})=j_{0,1}^2$ for the
Dirichlet Laplacian on the disk, \eqref{eq:AH} can be rewritten as the spectral
identity
\[
\|C_\mathbb{D}\|_{L^2(\mathbb{D})\to L^2(\mathbb{D})}=\frac{2}{\sqrt{\lambda_1(\mathbb{D})}}.
\]
This striking formula suggests a possible general principle: that the $L^2$ norm
of $C_D$ might be determined by the first Dirichlet eigenvalue $\lambda_1(D)$.

In \cite{Dostanic1996}, a natural extension of this principle was proposed: for
bounded domains with sufficiently regular boundary,
\begin{equation}\label{eq:dost-claim-intro}
\|C_D\|_{L^2(D)\to L^2(D)}=\frac{2}{\sqrt{\lambda_1(D)}}.
\end{equation}
Beyond this conjectural sharp-constant identity, \cite{Dostanic1996} contains
several valuable contributions on the spectral side of the subject, including
results and estimates connected to Dirichlet eigenvalues and their asymptotic
behavior, and it helped stimulate further interest in linking singular integral
bounds to Laplacian spectra. The present paper was directly motivated by this
spectral viewpoint, and our goal is to clarify precisely which parts of this
heuristic picture can be made rigorous.

The reasoning in \cite{Dostanic1996} implicitly uses the heuristic identification
\[
(-\Delta_D)^{-\alpha} f \stackrel{?}{=} \mathcal{F}^{-1}\!\big(|\xi|^{-2\alpha}\big)\,\mathcal{F}(f\mathbf{1}_D),
\]
i.e.\ it treats fractional Dirichlet powers of $-\Delta_D$ as Euclidean Fourier
multipliers after extending $f$ by zero outside $D$. As we explain in this paper,
this Fourier-multiplier formula does not hold in general on bounded domains
without additional structure, and relying on it can lead to incorrect sharp
constants.

\medskip
\noindent\textbf{What we prove.}
Our first contribution is an explicit endpoint counterexample (already for the
unit disk) to the Fourier-weighted inequality that underlies the proof of
\eqref{eq:dost-claim-intro}; see Section~\ref{22}. This pinpoints exactly where
the spectral identification breaks down.

Our second contribution is to give the correct sharp constant for the Fourier
weight $|\xi|^{-1}$ in full generality. We show that the quadratic form
\[
\int_{\R^2}\frac{|\widehat f(\xi)|^2}{|\xi|}\,d\xi
\]
is governed not by $\lambda_1(D)$, but by the positive self-adjoint integral
operator
\[
(S_D f)(x)=\int_D \frac{f(y)}{|x-y|}\,dy,
\]
and the best constant equals $(2\pi)^{-1}\|S_D\|$; see
Theorem~\ref{thm:sharp-constant}.

Finally, returning to the Cauchy transform itself, we show that the disk is the
unique domain for which the spectral value $2/\sqrt{\lambda_1(D)}$ can be
attained. More precisely, for a simply connected bounded domain $D$ with smooth
boundary and first Dirichlet eigenfunction $u$, we prove
\[
\frac{\|C_D u\|_{L^2(D)}}{\|u\|_{L^2(D)}}>\frac{2}{\sqrt{\lambda_1(D)}},
\]
and consequently
\[
\|C_D\|_{L^2(D)\to L^2(D)}>\frac{2}{\sqrt{\lambda_1(D)}}
\qquad\text{unless $D$ is a disk.}
\]

Remark that the inequality \[
\|C_D\|_{L^2(D)\to L^2(D)}\ge \frac{2}{\sqrt{\lambda_1(D)}},
\]  has been already proved in \cite{AndersonHinkkanen1989}.

Equivalently,
\[
\|C_D\|_{L^2(D)\to L^2(D)}=\frac{2}{\sqrt{\lambda_1(D)}} \quad\Longleftrightarrow\quad D \text{ is a disk}.
\]
At the heart of the proof is an $L^2$ orthogonal decomposition of $C_D u$ into
the $L^2$--minimal $\bar\partial$--solution and a nontrivial holomorphic
component, together with a rigidity argument showing that the holomorphic
component vanishes identically only in the disk case.

\medskip
\noindent\textbf{Organization.}
In Section~\ref{22} we present the Fourier-weighted counterexample on the disk. Moreover in the same section we identify the correct sharp constant via the operator $S_D$.
The final part of the paper, Section~\ref{five} establishes the strict inequality for the Cauchy
transform on noncircular domains and proves the disk rigidity statement.
\section{Background, motivation, and main results}\label{sec:intro}

Let $D\subset\C$ be a bounded domain. We study the (interior) Cauchy transform
\begin{equation}\label{eq:Cauchy-def}
(C_D f)(z):=\frac{1}{\pi}\int_D \frac{f(w)}{z-w}\,\dd A(w),\qquad z\in D,
\end{equation}
initially defined for $f\in C_0^\infty(D)$. The operator $C_D$ is a first-order
complex singular integral and is a natural $L^2$--solution operator for
$\bar\partial u=f$ on $D$, with connections to Bergman-type operators and
spectral questions (see, e.g., \cite{Dostanic1996,KalajVujadinovic2017CMB,Vujadinovic2025}).

\subsection{A Fourier-multiplier pitfall on bounded domains}\label{subsec:pitfall}

A key step in \cite{Dostanic1996} is to extend \(f\in L^2(D)\) by zero to \(\R^2\) and
then to treat fractional Dirichlet powers of \(-\Delta_D\) as though they were given by the
free-space Fourier multiplier \( |\xi|^{-2\alpha} \). This leads one to consider quadratic
forms of the type
\begin{equation}\label{eq:fourier-weighted-form}
\int_{\R^2}\frac{|\widehat f(\xi)|^2}{|\xi|^{2\alpha}}\,d\xi,
\end{equation}
which correspond to the operator \((-\Delta)^{-\alpha}\) acting on the zero extension.
However, \eqref{eq:fourier-weighted-form} is not, in general, the Dirichlet quadratic form
for \((-\Delta_D)^{-\alpha}\). In particular, there is no reason to expect the sharp bound
\begin{equation}\label{eq:wrong-fourier-ineq}
\int_{\R^2}\frac{|\widehat f(\xi)|^2}{|\xi|^{2\alpha}}\,d\xi
\;\le\; \lambda_1(D)^{-\alpha}\,\|f\|_{L^2(D)}^2,
\end{equation}
where \(\lambda_1(D)\) denotes the first Dirichlet eigenvalue.

At the endpoint $\alpha=\tfrac12$ this becomes the (wrong) inequality
\begin{equation}\label{eq:wrong-fourier-ineq-endpoint}
\int_{\R^2}\frac{|\widehat f(\xi)|^2}{|\xi|}\,d\xi
\;\le\; \lambda_1(D)^{-1/2}\,\|f\|_{L^2(D)}^2,
\end{equation}
and we give an explicit counterexample on the unit disk to the Fourier-weighted estimate
invoked in \cite[Lemma~2.2]{Dostanic1996}.

Motivated by the sharp disk result of Anderson--Hinkkanen, it was asserted in
\cite{Dostanic1996} that the $L^2$ operator norm of $C_D$ is determined exactly
by $\lambda_1(D)$:

\begin{theorem}[Claimed in \cite{Dostanic1996}]\label{thm:Dostanic-claimed}
Let $D\subset\C$ be a bounded domain with piecewise $C^1$ boundary. Then $C_D$
extends boundedly to $L^2(D)$ and
\begin{equation}\label{eq:Dostanic-claimed}
\|C_D\|_{L^2(D)\to L^2(D)}=\frac{2}{\sqrt{\lambda_1(D)}},
\end{equation}
where $\lambda_1(D)>0$ is the smallest eigenvalue of the Dirichlet problem
\begin{equation}\label{eq:Dirichlet-eigen}
\begin{cases}
-\Delta u=\lambda u & \text{in } D,\\
u|_{\partial D}=0. &
\end{cases}
\end{equation}
\end{theorem}

The proof in \cite{Dostanic1996} uses the Fourier-weighted inequality mentioned
in \S\ref{subsec:pitfall}. Since that inequality is false, the argument cannot
establish \eqref{eq:Dostanic-claimed}.

\section{A counterexample to the \eqref{eq:wrong-fourier-ineq}  at $\alpha=\tfrac12$ and the correct sharp constant for the $\abs{\xi}^{-1}$ Fourier weight}\label{22}

Throughout, we use the Fourier normalization
\[
\widehat f(\xi)=\frac1{2\pi}\int_{\R^2}e^{-ix\cdot\xi}f(x)\,\dd x,
\qquad
f(x)=\frac1{2\pi}\int_{\R^2}e^{ix\cdot\xi}\widehat f(\xi)\,\dd\xi .
\]
Let $D=\{x\in\R^2:\ |x|\le 1\}$ be the unit disk and set $f=\mathbf 1_D$.
Then
\[
\int_{|x|\le 1} e^{-i x\cdot \xi}\,\dd x
= \int_0^1\int_0^{2\pi} e^{-i r|\xi|\cos\theta}\,r\,\dd\theta\,\dd r
= 2\pi\int_0^1 J_0(r|\xi|)\,r\,\dd r
= 2\pi\,\frac{J_1(|\xi|)}{|\xi|},
\]
and hence
\begin{equation}\label{eq:ft-indicator}
\widehat f(\xi)=\frac{J_1(|\xi|)}{|\xi|}.
\end{equation}
Writing $r=|\xi|$ and using $\dd\xi=r\,\dd r\,\dd\theta$, we obtain
\begin{align}\label{eq:LHS2}
\int_{\R^2}\frac{\abs{\widehat f(\xi)}^2}{|\xi|}\,\dd\xi
&=\int_0^{2\pi}\int_0^\infty
\frac{\big(J_1(r)^2/r^2\big)}{r}\,r\,\dd r\,\dd\theta
=2\pi\int_0^\infty \frac{J_1(r)^2}{r^2}\,\dd r.
\end{align}
Using the classical identity
\[
\int_0^\infty \frac{J_1(r)^2}{r^2}\,\dd r=\frac{4}{3\pi},
\]
it follows that
\begin{equation}\label{eq:LHS}
\int_{\R^2}\frac{\abs{\widehat f(\xi)}^2}{|\xi|}\,\dd\xi=\frac{8}{3}.
\end{equation}
On the other hand,
\[
\int_D |f(x)|^2\,\dd x = |D|=\pi,
\qquad
\lambda_1(D)=j_{0,1}^2,
\]
so
\begin{equation}\label{eq:RHS}
\lambda_1(D)^{-1/2}\int_D |f(x)|^2\,\dd x
=\frac{\pi}{j_{0,1}}<\frac{8}{3}.
\end{equation}
Therefore
\[
\int_{\R^2}\frac{\abs{\widehat f(\xi)}^2}{|\xi|}\,\dd\xi
>\lambda_1(D)^{-1/2}\int_D \abs{f(x)}^2\,\dd x,
\]
and the estimate \cite[Lemma~2.2]{Dostanic1996}  fails on the unit disk for $\alpha=\tfrac12$.

Here, \(J_k\) denotes the Bessel function of the first kind of order \(k\), defined by the series
\[
J_k(x)=\sum_{m=0}^{\infty}\frac{(-1)^m}{m!\,(m+k)!}\left(\frac{x}{2}\right)^{2m+k},
\qquad k=0,1,2,\dots
\]

\(j_{0,1}\) denotes the first positive zero of \(J_0\), i.e.
\[
J_0\!\left(j_{0,1}\right)=0, \qquad 0<j_{0,1}<j_{0,2}<\cdots,
\]
and numerically
\[
j_{0,1}\approx 2.4048255577.
\]

\begin{theorem}[Sharp $L^2$ constant for the $|\xi|^{-1}$ Fourier weight]\label{thm:sharp-constant}
Let $D\subset\R^2$ be bounded and let $f\in L^2(D)$, extended by $0$ to $\R^2$.
Define the Fourier transform by
\[
\widehat f(\xi)=\frac1{2\pi}\int_{\R^2}e^{-ix\cdot\xi}f(x)\,\dd x,
\qquad
f(x)=\frac1{2\pi}\int_{\R^2}e^{ix\cdot\xi}\widehat f(\xi)\,\dd\xi .
\]
Let $S_D:L^2(D)\to L^2(D)$ be the integral operator
\[
(S_D f)(x)=\int_D\frac{f(y)}{|x-y|}\,\dd y,\qquad x\in D.
\]
Then $S_D$ is positive, self-adjoint and compact on $L^2(D)$ and for every
$f\in L^2(D)$,
\begin{equation}\label{eq:quadratic-form}
\int_{\R^2}\frac{|\widehat f(\xi)|^2}{|\xi|}\,\dd\xi
=\frac{1}{2\pi}\,\ip{S_D f}{f}_{L^2(D)}.
\end{equation}
Consequently, the sharp constant in
\[
\int_{\R^2}\frac{|\widehat f(\xi)|^2}{|\xi|}\,\dd\xi
\le C\,\|f\|_{L^2(D)}^2
\qquad (f\in L^2(D))
\]
is
\[
C_{\mathrm{sharp}}(D)=\frac{1}{2\pi}\,\|S_D\|_{L^2(D)\to L^2(D)}
=\frac{1}{2\pi}\,\lambda_{\max}(S_D).
\]
Moreover, by Schur's test,
\[
C_{\mathrm{sharp}}(D)\le
\frac{1}{2\pi}\sup_{x\in D}\int_D \frac{1}{|x-y|}\,\dd y.
\]
\end{theorem}

\begin{proof}
With the above normalization, Plancherel holds in the form
$\|f\|_{L^2(\R^2)}=\|\widehat f\|_{L^2(\R^2)}$, and the convolution identity is
$\widehat{g*h}=2\pi\,\widehat g\,\widehat h$.
Let $k(x)=|x|^{-1}$. In the sense of tempered distributions in $\R^2$,
\[
\widehat k(\xi)=\frac{1}{|\xi|}.
\]
Define $T f := \mathcal F^{-1}\!\big(|\xi|^{-1}\widehat f(\xi)\big)$. Then
\[
\widehat{Tf}(\xi)=\frac{1}{|\xi|}\widehat f(\xi)
=\widehat k(\xi)\,\widehat f(\xi)
=\frac{1}{2\pi}\widehat{k*f}(\xi),
\]
hence $Tf=\frac{1}{2\pi}(k*f)$. Since $f$ is supported in $D$,
\[
Tf(x)=\frac{1}{2\pi}\int_D\frac{f(y)}{|x-y|}\,\dd y=\frac{1}{2\pi}(S_D f)(x)
\quad\text{for a.e. }x\in D.
\]
Therefore, using Plancherel, let $g(\xi):=|\xi|^{-1}\widehat f(\xi)$. Then $\widehat{\mathcal F^{-1}g}=g$, so
\[
\big\langle \mathcal F^{-1}g,\,f\big\rangle_{L^2(\R^2)}
=\big\langle g,\,\widehat f\big\rangle_{L^2(\R^2)}
=\int_{\R^2}\frac{1}{|\xi|}\,\widehat f(\xi)\,\overline{\widehat f(\xi)}\,d\xi
=\int_{\R^2}\frac{|\widehat f(\xi)|^2}{|\xi|}\,d\xi.
\]
Thus

\[
\int_{\R^2}\frac{|\widehat f(\xi)|^2}{|\xi|}\,\dd\xi
=\ip{\mathcal F^{-1}\!\Big(\frac{1}{|\xi|}\widehat f\Big)}{f}_{L^2(\R^2)}
=\ip{Tf}{f}_{L^2(\R^2)}
=\frac{1}{2\pi}\ip{S_D f}{f}_{L^2(D)},
\]
which proves \eqref{eq:quadratic-form}.

The operator $S_D$ is bounded on $L^2(D)$ and is self-adjoint and positive since
its kernel $|x-y|^{-1}$ is real and symmetric. It is compact because
$|x-y|^{-1}\in L^2(D\times D)$ (the singularity is integrable in two dimensions),
so $S_D$ is Hilbert--Schmidt. Hence

Using \eqref{eq:quadratic-form}, the sharp constant in the Fourier inequality is
$C_{\mathrm{sharp}}(D)=\frac{1}{2\pi}\|S_D\|=\frac{1}{2\pi}\lambda_{\max}(S_D)$.

Finally, Schur's test gives
\[
\|S_D\|_{2\to2}\le \sup_{x\in D}\int_D\frac{1}{|x-y|}\,\dd y,
\]
and multiplying by $1/(2\pi)$ yields the stated bound for $C_{\mathrm{sharp}}(D)$. \paragraph{Schur's test (used here).}
Let $(Tf)(x)=\int_D K(x,y)f(y)\,dy$ with $K\ge 0$. Schur's test with the weight $m\equiv 1$ gives
\[
\|T\|_{2\to 2}\le
\Bigl(\sup_{x\in D}\int_D K(x,y)\,dy\Bigr)^{1/2}
\Bigl(\sup_{y\in D}\int_D K(x,y)\,dx\Bigr)^{1/2}.
\]
If $K(x,y)=K(y,x)$ is symmetric, then the two suprema coincide, hence
\[
\|T\|_{2\to 2}\le \sup_{x\in D}\int_D K(x,y)\,dy.
\]
Applying this to $S_D$ with $K(x,y)=1/|x-y|$ yields
\[
\|S_D\|_{2\to 2}\le \sup_{x\in D}\int_D\frac{1}{|x-y|}\,dy,
\]
and multiplying by $1/(2\pi)$ gives the corresponding bound for
$C_{\mathrm{sharp}}(D)=\frac{1}{2\pi}\|S_D\|.$
\end{proof}
\begin{remark} By the proof of the previous theorem, if $D=\mathbb{D}$ \[C\le \frac{1}{2\pi}\sup_{x\in D}\int_{\mathbb{D}}\frac{1}{|x-y|}\,dy=\frac{1}{2\pi}\int_{\mathbb{D}}\frac{1}{|y|}\,dy=1.
\]
It can be proved by using discrete methods and numerical analysis that the sharp constant $C$ in the inequality \[
\int_{\R^2}\frac{\abs{\widehat f(\xi)}^2}{|\xi|}\,\dd\xi
\le C\int_{\mathbb{D}} \abs{f(x)}^2\,\dd x,
\] for the unit disk $\mathbb{D}$ is equal to $C\approx 0.852$ which is larger than $8/(3\pi)\approx 0.848$ obtained from the concrete example above.
\end{remark}
\section{First-eigenfunction test and disk rigidity}\label{five}
\begin{theorem}\label{thm:strict-eigtest}
Let $D\subset\C$ be a bounded simply connected domain with piecewise $C^1$ boundary.
Let $u$ be the (real-valued) first Dirichlet eigenfunction on $D$, normalized arbitrarily, with
\[
-\Delta u=\lambda_1(D)\,u \quad \text{in }D,\qquad u|_{\partial D}=0.
\]
Set
\[
v_0:=-\frac{4}{\lambda_1(D)}\,\partial_z u,
\qquad\text{and}\qquad
w:=C_D u.
\]
Then $h:=w-v_0$ is holomorphic on $D$ and
\begin{equation}\label{eq:pythagoras}
\|w\|_{L^2(D)}^2=\|v_0\|_{L^2(D)}^2+\|h\|_{L^2(D)}^2.
\end{equation}
Moreover,
\begin{equation}\label{eq:v0-threshold}
\|v_0\|_{L^2(D)}=\frac{2}{\sqrt{\lambda_1(D)}}\,\|u\|_{L^2(D)}.
\end{equation}
Consequently,
\[
\frac{\|C_D u\|_{L^2(D)}}{\|u\|_{L^2(D)}}\ge \frac{2}{\sqrt{\lambda_1(D)}},
\]
with strict inequality whenever $h\not\equiv 0$ (equivalently, $C_Du\not\equiv v_0$). In particular,
\[
\|C_D\|_{L^2(D)\to L^2(D)}\ge \frac{2}{\sqrt{\lambda_1(D)}},
\]
and it is strictly larger if $D$ is not a disk (see Theorem~\ref{thm:rigidity-disk} below).
\end{theorem}

\begin{proof}
Since $\Delta=4\,\partial_z\partial_{\bar z}$ and $-\Delta u=\lambda_1 u$, we have
\[
\partial_{\bar z}v_0
=-\frac{4}{\lambda_1}\partial_{\bar z}\partial_z u
=-\frac{1}{\lambda_1}\Delta u
=u.
\]
On the other hand, $\partial_{\bar z}(C_D u)=u$ in $D$, hence
\[
\partial_{\bar z}(w-v_0)=0,
\]
so $h:=w-v_0$ is holomorphic on $D$.

Next, compute $\|v_0\|_2$. Since $u$ is real-valued, $|\partial_z u|^2=\frac14|\nabla u|^2$, and therefore
\[
\|v_0\|_{L^2(D)}^2
=\frac{16}{\lambda_1^2}\int_D |\partial_z u|^2\,dA
=\frac{4}{\lambda_1^2}\int_D |\nabla u|^2\,dA.
\]
By Green's identity and $u|_{\partial D}=0$,
\[
\int_D |\nabla u|^2\,dA=\int_D u(-\Delta u)\,dA=\lambda_1\|u\|_{L^2(D)}^2,
\]
which gives \eqref{eq:v0-threshold}.

Finally, we show $v_0\perp h$ in $L^2(D)$. Using $\overline{\partial_z u}=\partial_{\bar z}u$,
\[
\langle v_0,h\rangle_{L^2(D)}
=-\frac{4}{\lambda_1}\int_D (\partial_{\bar z}u)\,h\,dA.
\]
A $\bar\partial$--integration by parts yields
\[
\int_D (\partial_{\bar z}u)\,h\,dA
=\frac{i}{2}\int_{\partial D}u\,h\,dz-\int_D u\,(\partial_{\bar z}h)\,dA=0,
\]
because $u|_{\partial D}=0$ and $\partial_{\bar z}h=0$. Hence $\langle v_0,h\rangle=0$, and
\eqref{eq:pythagoras} follows. The stated inequalities are immediate from
\eqref{eq:pythagoras} and \eqref{eq:v0-threshold}.
\end{proof}


\begin{lemma}[Exterior holomorphy and decay of the Cauchy transform]\label{lem:ext-holo-decay}
Let $D\subset\C$ be bounded and let $f\in L^1(D)$. Define
\[
F(z):=\frac1\pi\int_D \frac{f(w)}{z-w}\,dA(w),\qquad z\in\Omega:=\C\setminus\overline D.
\]
Then $F$ is holomorphic on $\Omega$ and, as $z\to\infty$,
\[
F(z)=\frac{1}{\pi}\left(\frac{M_0}{z}+\frac{M_1}{z^2}+\frac{M_2}{z^3}+\cdots\right),
\qquad
M_k:=\int_D f(w)\,w^k\,dA(w),
\]
in particular $F(z)=O(1/z)$.
\end{lemma}

\begin{proof}
Let \(D\subset\mathbb C\) be measurable and bounded, \(f\in L^{1}(D)\), and define
\[
F(z):=\frac1\pi\int_D \frac{f(w)}{z-w}\,dA(w),\qquad z\in \Omega:=\mathbb C\setminus \overline D.
\]

Fix \(z_0\in\Omega\). Set
\[
\delta:=\dist(z_0,D)>0.
\]
Choose \(r\in(0,\delta/2)\). Then for all \(z\in B(z_0,r)\) and all \(w\in D\),
\[
|z-w|\ge |z_0-w|-|z-z_0|\ge \delta-r \ge \delta/2.
\]
Hence the kernels are uniformly dominated on \(B(z_0,r)\):
\[
\Big|\frac{f(w)}{z-w}\Big|\le \frac{2|f(w)|}{\delta},\qquad
\Big|\frac{f(w)}{(z-w)^2}\Big|\le \frac{4|f(w)|}{\delta^2},
\]
and the dominating functions \(2|f|/\delta\), \(4|f|/\delta^2\) are in \(L^{1}(D)\).

For each fixed \(w\), the map \(z\mapsto (z-w)^{-1}\) is holomorphic on \(B(z_0,r)\) with derivative
\(- (z-w)^{-2}\). The domination above allows differentiation under the integral sign (e.g. via dominated
convergence applied to difference quotients), giving
\[
F'(z)= -\frac1\pi\int_D \frac{f(w)}{(z-w)^2}\,dA(w),\qquad z\in B(z_0,r).
\]
Since \(z_0\) was arbitrary, \(F\) is holomorphic on \(\Omega\).

Let \(D\) be bounded and pick \(R>0\) such that \(D\subset\{\,|w|\le R\,\}\). If \(|z|>2R\), then
\(|w/z|\le 1/2\) for \(w\in D\), and
\[
\frac1{z-w}=\frac1z\,\frac1{1-w/z}
=\frac1z\sum_{k=0}^\infty \Big(\frac{w}{z}\Big)^k,
\]
with absolute (and uniform in \(w\in D\)) convergence.

Moreover,
\[
\sum_{k=0}^\infty \Big|\frac{w^k}{z^{k+1}}\Big|
\le \frac1{|z|}\sum_{k=0}^\infty \Big(\frac{R}{|z|}\Big)^k
\le \frac1{|z|}\cdot \frac1{1-1/2}=\frac{2}{|z|}.
\]
Thus
\[
\Big|\frac{f(w)}{z-w}\Big|
\le |f(w)|\sum_{k=0}^\infty \Big|\frac{w^k}{z^{k+1}}\Big|
\le \frac{2|f(w)|}{|z|},
\]
which is integrable over \(D\). This justifies termwise integration:
\[
F(z)=\frac1\pi\sum_{k=0}^\infty \frac{1}{z^{k+1}}\int_D f(w)\,w^k\,dA(w)
=\frac1\pi\left(\frac{M_0}{z}+\frac{M_1}{z^2}+\frac{M_2}{z^3}+\cdots\right),
\]
where
\[
M_k:=\int_D f(w)\,w^k\,dA(w).
\]
In particular,
\[
|F(z)|\le \frac1\pi\int_D \frac{|f(w)|}{|z|-|w|}\,dA(w)
\le \frac1\pi\cdot \frac{1}{|z|-R}\,\|f\|_{L^1(D)}
=O\!\left(\frac1{|z|}\right),
\]
so \(F(z)=O(1/z)\) as \(z\to\infty\), and the Laurent series above is valid for \(|z|>2R\).

If \(M_0=M_1=\cdots=M_{m-1}=0\), then the expansion starts at \(z^{-(m+1)}\) and
\(F(z)=O(|z|^{-(m+1)})\).

\end{proof}


\begin{theorem}[Rigidity]\label{thm:rigidity-disk}
Let $D \subset \C$ be a bounded, simply connected domain with $C^{1,\alpha}$ boundary.
Let $u$ be the first Dirichlet eigenfunction,
\[
-\Delta u=\lambda_1(D)\,u \quad \text{in } D, \qquad u=0 \quad \text{on } \partial D,
\]
and define
\[
v_0 := -\frac{4}{\lambda_1(D)}\,\partial_z u .
\]
Then
\[
C_D u \equiv v_0 \quad \text{in } D
\]
if and only if $D$ is a disk.
\end{theorem}

\begin{proof}
Assume first that $C_Du\equiv v_0$ in $D$, where $v_0=-(4/\lambda_1)\partial_z u$.
Let $\Omega:=\C\setminus\overline D$ and consider the exterior Cauchy transform
\[
F(z):=\frac{1}{\pi}\int_D \frac{u(w)}{z-w}\,dA(w),\qquad z\in\Omega.
\]
By Lemma~\ref{lem:ext-holo-decay}, $F$ is holomorphic on $\Omega$ and satisfies $F(z)=O(1/z)$ as $z\to\infty$.

For smooth $u$ and $\partial D$, the same integral defines a continuous function on $\C$, and its restriction
to $D$ coincides with $C_Du$. Hence the identity $C_Du=v_0$ on $D$ implies that $v_0$ has a holomorphic
extension to $\Omega$, given by
\[
\partial_z u(w) := -\frac{\lambda_1}{4}\,F(w),\qquad w\in\Omega.
\]
In particular, this extension is holomorphic on $\Omega$ and
\begin{equation}\label{eq:dz-decay}
\partial_z u(w)=O(1/w)\qquad (w\to\infty).
\end{equation}

\paragraph{Boundary identity.}
On $\partial D$ we have $u=0$, so $\nabla u$ is normal:
\[
\nabla u = (\partial_n u)\,n,
\]
where $n=n_x+i n_y$ is the unit outward normal. Using $\partial_z=\tfrac12(\partial_x-i\partial_y)$, we obtain
\[
\partial_z u
=\frac12(\partial_n u)\,\overline{n}
\qquad \text{on }\partial D.
\]
Let $\gamma=\gamma(s)$ be an arclength parametrization of $\partial D$. Then $|\gamma'(s)|=1$ and
\[
n(s)=-i\,\gamma'(s),
\qquad\text{hence}\qquad
\overline{n(s)}=i\,\overline{\gamma'(s)}=\frac{i}{\gamma'(s)}.
\]
Consequently,
\[
\partial_z u(\gamma(s))\,d\gamma
=\partial_z u(\gamma(s))\,\gamma'(s)\,ds
=\frac12(\partial_n u)\,\overline{n}\,\gamma'(s)\,ds
=\frac{i}{2}(\partial_n u)\,ds,
\]
so the complex differential $\partial_z u\,d\gamma$ has vanishing real part on $\partial D$
(equivalently, it is purely imaginary along $\partial D$).

\paragraph{Rigidity via the exterior conformal map.}
Let $\Phi:\Omega\to\{|\zeta|>1\}$ be the exterior conformal map normalized by $\Phi(\infty)=\infty$,
and write
\[
Z(\zeta):=\Phi^{-1}(\zeta).
\]
Define
\[
\omega(\zeta):=\bigl(\partial_z u\bigr)\bigl(Z(\zeta)\bigr)\,Z'(\zeta),
\qquad |\zeta|>1.
\]
Then $\omega$ is holomorphic for $|\zeta|>1$, and \eqref{eq:dz-decay} together with the normalization
\[
Z(\zeta)=a\zeta+b+O(1/\zeta)\qquad (\zeta\to\infty)
\]
implies
\[
\omega(\zeta)=O(1/\zeta)\qquad (\zeta\to\infty),
\]
so $F_1(\zeta):=\zeta\,\omega(\zeta)$ is holomorphic on $|\zeta|>1$ and bounded at $\infty$.

On $|\zeta|=1$ we have $Z(\zeta)\in\partial D$ and $dZ=Z'(\zeta)\,d\zeta$, hence
\[
\omega(\zeta)\,d\zeta
=\bigl(\partial_z u\bigr)\bigl(Z(\zeta)\bigr)\,dZ.
\]
The boundary identity above shows that this differential has vanishing real part on $|\zeta|=1$; therefore,
for $\zeta=e^{it}$ (so $d\zeta=i\zeta\,dt$) we obtain
\[
\Re\!\bigl(F_1(e^{it})\,i\,dt\bigr)=0
\qquad\Longrightarrow\qquad
F_1(e^{it})\in\R,\quad t\in\R.
\]
Thus $F_1$ extends holomorphically across $|\zeta|=1$ by Schwarz reflection and is an entire bounded
function; hence $F_1$ is constant, say $F_1(\zeta)\equiv c\in\R$, and thus
\[
\omega(\zeta)=\frac{c}{\zeta}.
\]
Unwinding the definition of $\omega$ shows that $|Z'(\zeta)|$ is constant on $|\zeta|=1$.
Since $\log|Z'(\zeta)|$ is harmonic on $|\zeta|>1$, constant boundary values imply $\log|Z'|$ is constant
on $|\zeta|>1$, so $Z'(\zeta)$ is constant and $Z(\zeta)$ is a M\"obius map. Consequently $\partial D$ is a circle
and $D$ is a disk.

\noindent\textbf{Conversely.}
If $D$ is a disk, the first eigenfunction is radial and a direct computation (e.g.\ in polar coordinates)
shows that $C_D u = -\frac{4}{\lambda_1(D)}\,\partial_z u$ holds identically.
\end{proof}
\begin{remark}
The rigidity theorem above  is formulated for smooth domains (e.g.\ $C^{1,\alpha}$) in order to use
boundary identities and conformal mapping. The next example has a different role: it is a purely
interior computation on the unit square, showing directly that $C_Du\not\equiv v_0$ by evaluating at
a single point. In particular, it yields
\[
\frac{\|C_Du\|_{L^2(D)}}{\|u\|_{L^2(D)}}>\frac{2}{\sqrt{\lambda_1(D)}},
\]
and hence $\|C_D\|_{L^2(D)\to L^2(D)}>\frac{2}{\sqrt{\lambda_1(D)}}$. Thus the claimed norm identity \eqref{eq:dost-claim-intro}
cannot hold for the square.
\end{remark}
\begin{example}\label{ex:square}
Let $D=(0,1)^2$, for which $\lambda_1(D)=2\pi^2$, and take
\[
u(x,y)=\sin(\pi x)\sin(\pi y).
\]
Then $u_z(0,0)=0$, hence $v_0(0)=0$. On the other hand, evaluating the Cauchy transform at $z=0$ gives
\[
(C_Du)(0)=-\frac1\pi\int_0^1\int_0^1\frac{\sin(\pi x)\sin(\pi y)}{x+iy}\,dy\,dx,
\]
and writing $(x+iy)^{-1}=(x-iy)/(x^2+y^2)$ shows
\[
\Re\,(C_Du)(0)
=-\frac1\pi\int_0^1\int_0^1\frac{x\,\sin(\pi x)\sin(\pi y)}{x^2+y^2}\,dy\,dx<0.
\]
Thus $(C_Du)(0)\neq v_0(0)$, hence $h\not\equiv 0$ in Theorem~\ref{thm:strict-eigtest}, and therefore
\[
\frac{\|C_Du\|_{L^2(D)}}{\|u\|_{L^2(D)}}>\frac{2}{\sqrt{\lambda_1(D)}}=\frac{\sqrt2}{\pi}.
\]
\end{example}
We finish this note by stating a proving a similar rigity result for a circular annuli
\begin{theorem}
Let $A=A(r,R)=\{z\in\C:\ r<|z|<R\}$ with $0<r<R$, and let $f$ be the first Dirichlet eigenfunction of
$-\Delta$ on $A$, i.e.
\[
-\Delta f=\lambda_1(A)\,f\quad \text{in }A,\qquad f=0\quad \text{on }\partial A,
\]
normalized arbitrarily and chosen positive in $A$. Define the  Cauchy transform
\[
C_A g(z):=\frac1\pi\int_A \frac{g(w)}{z-w}\,d\mu(w),\qquad d\mu=du\,dv,
\]
and set
\[
v_0:=-\frac{4}{\lambda_1(A)}\,\partial_z f.
\]
Then $f$ is radial and
\[
C_A f(z)=\frac{2}{z}\int_{r}^{|z|} f(\rho)\,\rho\,d\rho,\qquad z\in A.
\]
Moreover, one has the exact identity
\[
C_A f(z)= -\frac{4}{\lambda_1(A)}\,\partial_z f(z)\;+\;\frac{2r\,f'(r)}{\lambda_1(A)}\,\frac{1}{z},
\qquad z\in A,
\]
and the orthogonality relation
\[
\left\langle v_0,\frac1z\right\rangle_{L^2(A)}=0.
\]
Consequently,
\[
\|C_A f\|_{L^2(A)}^2=\|v_0\|_{L^2(A)}^2
+\left|\frac{2r\,f'(r)}{\lambda_1(A)}\right|^2\left\|\frac1z\right\|_{L^2(A)}^2
>\|v_0\|_{L^2(A)}^2,
\]
and since $\Delta=4\partial_z\partial_{\bar z}$ one has
\[
\|v_0\|_{L^2(A)}^2
=\frac{16}{\lambda_1(A)^2}\|\partial_z f\|_{L^2(A)}^2
=\frac{16}{\lambda_1(A)^2}\cdot\frac14\|\nabla f\|_{L^2(A)}^2
=\frac{4}{\lambda_1(A)}\|f\|_{L^2(A)}^2.
\]
Therefore,
\[
\frac{\|C_A f\|_{L^2(A)}}{\|f\|_{L^2(A)}}
>\frac{\|v_0\|_{L^2(A)}}{\|f\|_{L^2(A)}}
=\frac{2}{\sqrt{\lambda_1(A)}}.
\]
\end{theorem}

\begin{proof}
By simplicity of $\lambda_1(A)$ and rotational invariance of $A$, the first eigenfunction is radial.
For radial $f$, write $z=\rho e^{i\theta}$ and $w=se^{i\varphi}$. Then
\[
C_A f(z)=\frac1\pi\int_r^R\int_0^{2\pi}\frac{f(s)\,s}{\rho e^{i\theta}-s e^{i\varphi}}\,d\varphi\,ds.
\]
A residue computation yields
\[
\int_0^{2\pi}\frac{d\varphi}{\rho e^{i\theta}-s e^{i\varphi}}
=
\begin{cases}
\displaystyle \frac{2\pi}{\rho e^{i\theta}}=\frac{2\pi}{z}, & s<\rho,\\[4pt]
0, & s>\rho,
\end{cases}
\]
hence
\[
C_A f(z)=\frac{2}{z}\int_r^{|z|} f(s)\,s\,ds.
\]
In particular, $\lim_{|z|\downarrow r}C_A f(z)=0$. Set $\lambda=\lambda_1(A)$ and
$v_0=-(4/\lambda)\partial_z f$. Using $\Delta=4\partial_z\partial_{\bar z}$ and $-\Delta f=\lambda f$,
we have $\bar\partial v_0=f$ in $A$. Also $\bar\partial(C_A f)=f$ in $A$, hence
$h:=C_A f-v_0$ is holomorphic in $A$. Since $f$ is radial, $h(e^{it}z)=e^{-it}h(z)$, thus $h(z)=c/z$
for some $c\in\C$. Multiplying by $z$ and letting $|z|\downarrow r$ gives
\[
c=\lim_{|z|\downarrow r}\bigl(zC_A f(z)-zv_0(z)\bigr)=-\lim_{|z|\downarrow r}zv_0(z).
\]
For radial $f$,
\[
\partial_z f(\rho e^{i\theta})=\frac{e^{-i\theta}}{2}f'(\rho),
\qquad\Rightarrow\qquad
zv_0(z)=-\frac{2\rho}{\lambda}f'(\rho),
\]
hence $c=(2r/\lambda)f'(r)$ and
\[
C_A f=v_0+\frac{c}{z}.
\]
Next,
\[\begin{split}
\left\langle v_0,\frac1z\right\rangle
&=\int_0^{2\pi}\int_r^R
\left(-\frac{2}{\lambda}e^{-i\theta}f'(\rho)\right)
\left(\frac{e^{i\theta}}{\rho}\right)\rho\,d\rho\,d\theta
\\&=-\frac{4\pi}{\lambda}\int_r^R f'(\rho)\,d\rho
=-\frac{4\pi}{\lambda}\bigl(f(R)-f(r)\bigr)=0.
\end{split}
\]
Therefore $v_0\perp 1/z$ in $L^2(A)$ and
\[
\|C_A f\|_2^2=\|v_0\|_2^2+|c|^2\left\|\frac1z\right\|_2^2.
\]
Since $f'(r)\neq 0$ by Hopf's lemma (because $\Delta f<0$), $c\neq 0$, and moreover
\[
\left\|\frac1z\right\|_{L^2(A)}^2=\int_A \frac{1}{|z|^2}\,dA
=2\pi\int_r^R \frac{1}{\rho}\,d\rho=2\pi\log\frac{R}{r}>0,
\]
we obtain $\|C_A f\|_2^2>\|v_0\|_2^2$.

Finally,
\[
\|\partial_z f\|_2^2=\frac14\|\nabla f\|_2^2
=\frac14\lambda\|f\|_2^2,
\]
which yields $\|v_0\|_2^2=(4/\lambda)\|f\|_2^2$ and the stated strict inequality.
\end{proof}

\begin{remark}
It is also worth emphasizing that there is a closely related Cauchy-type transform arising from the
Dirichlet problem for the Laplacian for which the sharp constant \(\frac{2}{\sqrt{\lambda_1(\Omega)}}\)
\emph{does} hold on arbitrary bounded domains. Following Astala--Iwaniec--Martin, let \(G(z,\tau)\) denote
the Green function of a bounded domain \(\Omega\), and for \(\varphi\in L^2(\Omega)\) define the Green
potential
\begin{equation}\label{eq:green-potential}
v=L_\Delta(\varphi)=\int_\Omega G(z,\tau)\,\varphi(\tau)\,d\tau,
\end{equation}
so that \(v|_{\partial\Omega}=0\) in the sense of distributions. The associated Cauchy transform is
\begin{equation}\label{eq:assoc-cauchy}
\mathcal{C}_\Delta(\varphi)(z)=v_z=\int_\Omega G_z(z,\tau)\,\varphi(\tau)\,d\tau .
\end{equation}
It was proved in \cite{Kalaj2012IEOT} that
\begin{equation}\label{eq:kalaj-norm}
\|\mathcal{C}_\Delta\|_{L^2(\Omega)\to L^2(\Omega)}=\frac{2}{\sqrt{\lambda_1(\Omega)}},
\end{equation}
where \(\lambda_1(\Omega)\) is the first Dirichlet eigenvalue of \(-\Delta\) on \(\Omega\).
\end{remark}

\noindent\textbf{Further background.}
For spectral/norm estimates for the interior Cauchy transform and its relation to Laplace/Dirichlet
spectral data (including sharp constants and asymptotics), see
\cite{ArazyKhavinson1992,Dostanic1998,Dostanic2005IEOT,Dostanic2007JFA,Liu2016,Khavinson2018Conm720}.
For the potential-theoretic variational viewpoint (Poincar\'e-type problems) closely related to
these operator norms, see \cite{KhavinsonPutinarShapiro2007}.
A general reference on the Cauchy transform and its role in planar potential theory is \cite{Bell2015}.
For recent singular-number asymptotics results see \cite{Vujadinovic2025} .
\section*{Acknowledgements}
The  author gratefully acknowledges financial support from the Ministry of Education, Science and Innovation of Montenegro through the grants \emph{``Mathematical Analysis, Optimisation and Machine Learning''} and \emph{``Complex-analytic and geometric techniques for non-Euclidean machine learning: theory and applications.''} He is also grateful to D.~Khavinson and P.~Melentijevi\'c for useful comments and encouragement.
 \paragraph{Competing interests.}
The author declares that there are no competing interests.

\paragraph{\bf Data availability.}
Data sharing is not applicable to this article as no datasets were generated or analysed during the current study.

\paragraph{\bf Ethical approval.}
Not applicable.

\paragraph{\bf Consent to participate.}
Not applicable.

\paragraph{\bf Consent for publication.}
Not applicable.

\end{document}